\documentclass[11pt]{amsart}
\usepackage{amsmath,amssymb,latexsym,esint,cite,mathrsfs}
\usepackage{verbatim,wasysym,cite}
\usepackage[left=2.7cm,right=2.7cm,top=2.9cm,bottom=2.9cm]{geometry}
\usepackage[latin1]{inputenc}
\usepackage{microtype}
\usepackage{color,enumitem,graphicx}
\usepackage[colorlinks=true,urlcolor=blue, citecolor=red,linkcolor=blue,
linktocpage,pdfpagelabels, bookmarksnumbered,bookmarksopen]{hyperref}
\usepackage[hyperpageref]{backref}
\usepackage[english]{babel}

\renewcommand{\epsilon}{{\varepsilon}}

\numberwithin{equation}{section}
\newtheorem{theorem}{Theorem}[section]
\newtheorem{lemma}[theorem]{Lemma}
\newtheorem{remark}[theorem]{Remark}

\newtheorem{definition}[theorem]{Definition}
\newtheorem{proposition}[theorem]{Proposition}

\title[Logarithmic NLS equation with harmonic potential]
{Logarithmic Bose-Einstein condensates \\ with harmonic potential}

\author[A.H. Ardila]{Alex H. Ardila}
\author[L. Cely]{Liliana Cely}
\author[M. Squassina]{Marco Squassina}

\address[A.H. Ardila]{Instituto Nacional de Matem\'atica Pura e Aplicada - IMPA
	\newline\indent
Estrada Dona Castorina, 110 CEP 22460-320
	\newline\indent
Rio de Janeiro, RJ - Brasil}
\email{ardila@impa.br}

\address[L. Cely]{Instituto de Matemática e Estatística \newline\indent
Departamento de Matemática, Universidade de São Paulo,	05508-090 \newline\indent
São Paulo, SP, Brazil	}
\email{mlcelyp@ime.usp.br}

\address[M.\ Squassina]{Dipartimento di Matematica e Fisica \newline\indent
	Universit\`a Cattolica del Sacro Cuore \newline\indent
	Via dei Musei 41, I-25121 Brescia, Italy}
\email{marco.squassina@unicatt.it}


\thanks{The first author was partially supported by CNPq/Brazil, through grant No. 152672/2016-8. 
The third author is member
of {\em Gruppo Nazionale per l'Analisi Ma\-te\-ma\-ti\-ca, la Probabilit\`a e le loro Applicazioni} (GNAMPA)}

\begin{document}

\begin{abstract}
In this paper, by using a compactness method, we study the Cauchy problem of the logarithmic Schr\"{o}dinger equation with harmonic potential. 
We then address the existence of ground states solutions as minimizers of the action on the Nehari manifold. Finally, we explicitly compute ground states (Gausson-type solution) and we show their orbital stability.
\end{abstract}

\maketitle

\begin{center}
	\begin{minipage}{8.5cm}
		\small
		\tableofcontents
	\end{minipage}
\end{center}

\medskip

\section{Introduction}
\label{S:0}
Recently, Zloshchastiev \cite{ADL} introduced a new Bose-Einstein condensate in a harmonic trap as a candidate structure of physical vacuum, this structure is described by a logarithmic nonlinear Schr\"odinger equation in presence of a harmonic potential.  The main motivation of such condensates lies essentially in their important applications  in quantum mechanics,  nuclear physics, quantum optics. 
Extensive details of the physical problem related to logarithmic Bose-Einstein condensate, experimental data and previous numerical studies can be found in \cite{BBA} and the references therein.

The aim of this work is the study of the existence and stability of the ground states associated with the following nonlinear Schr\"{o}dinger equation 
\begin{equation}
\label{0NL}
  i\partial_{t}u+\Delta u -V(x)u+\,u\,\mathrm{Log}\,\left|u\right|^{2}=0, \quad (x,t)\in \mathbb{R}^{N}\times\mathbb{R},
\end{equation}
where $t$ is time, $x\in \mathbb{R}^{N}$ is the spatial coordinate ($N\geq1$) and $u := u(x; t)\in \mathbb{C}$  is the  wave function. 
The local term $u\,\mathrm{Log}\,\left|u\right|^{2}$ describes the short-range interaction forces between particles. The potential $V (x)$   describes an electromagnetic field and has the following harmonic confinement
\begin{equation*} 
V(x)=\gamma(\gamma-1)|x|^{2}, \quad \gamma>1.
\end{equation*}

In the absence of the  harmonic potential, Eq. \eqref{0NL} now arises from different applications in quantum mechanics,  nuclear physics \cite{HE}, open quantum systems and  Bose-Einstein condensation.  We refer the readers to \cite{ADL, APLES, BBA, MFG}  for more information on the related physical backgrounds.  The classical logarithmic NLS equation  was proposed by Bialynicki-Birula and Mycielski \cite{CAS} as a model of nonlinear wave mechanics.

To the best of our knowledge, existence and stability of the ground states of logarithmic NLS equation \eqref{0NL} in presence of a harmonic potential has not been studied in the literature. More precisely,  Eq. \eqref{0NL} has been previously studied only for $\gamma=1$ (without the term $V (x)$). Among such works, let us mention \cite{ArdSqu,CB,TA,CL, CALO, PHBJ, PHST, AHA1}. This type of equations have been of great interest to both the theoretical and applied literature in recent years, see \cite{CZ, AHAA}.

Concerning the Schr\"{o}dinger equation with power-type nonlinearities and harmonic potential,  many authors have been studying the
problem of existence and stability of standing waves, see for instance \cite{JZ2000, Carles2002, FUKUIOHA2003, Fukui2000,FOG,JZ2SS005} and the references therein.

The many-dimensional harmonic oscillator $-\Delta+\gamma(\gamma-1)|x|^{2}$ is a self-adjoint operator on $L^{2}(\mathbb{R}^{N})$ with  operator domain $\left\{u\in H^{2}(\mathbb{R}^{N}): |x|^{2}u\in L^{2}(\mathbb{R}^{N})\right\}$ and quadratic form domain 
\begin{equation*} 
\Sigma(\mathbb{R}^{N})=\left\{u\in H^{1}(\mathbb{R}^{N}):|x|^{}u\in L^{2}(\mathbb{R}^{N}) \right\}.
\end{equation*}
It is well known that $\Sigma(\mathbb{R}^{N})$ is a Hilbert space when is equipped with the norm 
\begin{equation*}
\|u\|^{2}_{\Sigma}=\int_{\mathbb{R}^{N}}\left(\left|\nabla u\right|^{2}+|x|^{2}|u|^{2}\right)dx,
\end{equation*}
and it is continuously embedded in $H^{1}(\mathbb{R}^{N})$ due to the Hardy inequality. Along the flow of \eqref{0NL}, we have the conservation of the $L^{2}$-norm and of the energy functional associated:
\begin{equation}\label{EEEE}
E(u)=\frac{1}{2}\int_{\mathbb{R}^{N}}\left|\nabla u\right|^{2}dx+\frac{\gamma(\gamma-1)}{2}\int_{\mathbb{R}^{N}}|x|^{2}|u|^{2}dx  -\frac{1}{2}\int_{\mathbb{R}^{N}}\left|u\right|^{2}\mbox{Log}\left|u\right|^{2}dx.
\end{equation}
Note that $E$ is the generating Hamiltonian of \eqref{0NL}. It is important to note that the logarithmic nonlinearity $z\rightarrow z\,\mathrm{Log}\,\left|z\right|^{2}$ is not locally Lipschitz continuous due to the singularity of the logarithm at the origin, in particular $E\notin C^{1}(H^{1}(\mathbb{R}^{N}))$. In Section \ref{S:1},  we will show that the energy  $E$ is well-defined and of class $C^1$ on the energy space  $\Sigma(\mathbb{R}^{N})$, which implies that if $u\in C(\mathbb{R}, \Sigma(\mathbb{R}^{N}))\cap C^{1}(\mathbb{R}, \Sigma^{\prime}(\mathbb{R}^{N}))$, then  Eq.\eqref{0NL} makes sense in the space $\Sigma^{\prime}(\mathbb{R}^{N})$. Here, $\Sigma^{\prime}(\mathbb{R}^{N})$ is the dual space of $\Sigma^{}(\mathbb{R}^{N})$.

We have the following result concerning the well-posedness of the Cauchy problem for \eqref{0NL} in the energy space $\Sigma(\mathbb{R}^{N})$. The proof is done in Section \ref{S:2}.

\begin{proposition}[Well posedness] 
	\label{PCS}
Assume that $\gamma>1$. Then the Cauchy problem for \eqref{0NL} is globally well posed in  the energy space $\Sigma(\mathbb{R}^{N})$, i.e for every $u_{0}\in \Sigma^{}(\mathbb{R}^{N})$, there is a unique global solution  $u\in C(\mathbb{R},\Sigma^{}(\mathbb{R}^{N}))\cap C^{1}(\mathbb{R}, \Sigma^{\prime}(\mathbb{R}^{N}))$  with $u(0)=u_{0}$. In addition, the conservation of energy and charge hold, that is 
\begin{equation*}
E(u(t))=E(u_{0})\quad  and \quad \left\|u(t)\right\|^{2}_{L^{2}}=\left\|u_{0}\right\|^{2}_{L^{2}},\quad  \text{for all $t\in \mathbb{R}$}.
\end{equation*}
\end{proposition}
The most important issue in view of the applications of \eqref{0NL} in atomic physics and
quantum optics seems to be the study of standing waves solutions of \eqref{0NL}. In this case they are solutions  of \eqref{0NL} of the form $u(x,t)=e^{i\omega t}\varphi(x)$, where $\omega\in \mathbb{R}$ and $\varphi$ is a real valued function which has to solve the following nonlinear scalar field equation
\begin{equation}\label{EP}
-\Delta \varphi+\omega \varphi+\gamma(\gamma-1)|x|^{2}\varphi-\varphi\, \mathrm{Log}\left|\varphi \right|^{2}=0, \quad x\in {\mathbb{R}}^{N}.
\end{equation}

Before stating our results, we introduce some notations to be used throughout the paper. For $\omega\in \mathbb{R}$ and $\gamma>0$, we define the following functionals of class $C^{1}$ on $\Sigma^{}(\mathbb{R}^{N})$:
\begin{align*}
 S_{\omega}(u)&:=\frac{1}{2}\|\nabla u\|_{L^{2}}^{2}+\frac{\gamma(\gamma-1)}{2}\|xu\|_{L^{2}}^{2} +\frac{\omega+1}{2}\|u\|_{L^{2}}^{2}-\frac{1}{2}\int_{\mathbb{R}^{N}}\left|u\right|^{2}\mbox{Log}\left|u\right|^{2}dx,\\
  I_{\omega}(u)&:=\|\nabla u\|_{L^{2}}^{2}+\gamma(\gamma-1)\|xu\|_{L^{2}}^{2}+\omega\|u\|_{L^{2}}^{2}-\int_{\mathbb{R}^{N}}\left|u\right|^{2}\mbox{Log}\left|u\right|^{2}dx.
\end{align*}

Note the scalar field equation \eqref{EP} is varational in natura, that is, any solution is a critical point of $S_{\omega}(u)$.
It is not difficult to show that $I_{\omega}(u)=\left\langle S_{\omega}^{\prime}(u),u\right\rangle$.

For Eq. \eqref{0NL}, the ground state solution play a crucial role in the dynamics. We recall that a nontrivial solution $\varphi\in \Sigma^{}(\mathbb{R}^{N})$ of \eqref{EP} is termed as a ground state if it has some minimal action among all solutions of the nonlinear scalar field equation \eqref{EP}. In particular, it is possible to prove existence of ground state solutions solving the constrained variational problem
\begin{align}
\begin{split}\label{MPE}
d(\omega)&={\inf}\left\{S_{\omega}(u):\, u\in \Sigma^{}(\mathbb{R}^{N}) \setminus  \left\{0 \right\},  I_{\omega}(u)=0\right\} \\ 
&=\frac{1}{2}\,{\inf}\left\{\left\|u\right\|_{L^{2}}^{2}:u\in  \Sigma^{}(\mathbb{R}^{N})\setminus \left\{0 \right\},  I_{\omega}(u)= 0 \right\}. \end{split}
\end{align}
We define the set of ground states by
\begin{equation*}
 \mathcal{G}_{\omega}=\bigl\{ \varphi\in \Sigma^{}(\mathbb{R}^{N}) \setminus  \left\{0 \right\}: S_{\omega}(\varphi)=d(\omega), \,\, I_{\omega}(\varphi)=0\bigl\}.
\end{equation*}
In Section \ref{S:3}, we show that the quantity $d(\omega)$ is positive for every $\omega\in \mathbb{R}$.
Indeed, for all $\gamma>1$, $N\in \mathbb{N}$ and $\omega\in \mathbb{R}$ one has that
\begin{equation*}
d(\omega)=\frac{1}{2}{{\pi}^{\frac{N}{2}}{\gamma}^{-\frac{N}{2}}}e^{\omega+\gamma N}. 
\end{equation*}
Moreover, we have that any minimizing sequence is compact,  the minimum is achieved and we explicitly compute the ground states of  \eqref{EP}.
More precisely,

\begin{theorem}[Ground states] 
	\label{ESSW}
Let $N\geq 1$, $\gamma>1$ and $\omega\in \mathbb{R}$. Then we have\\
{\rm (i)} Any minimizing sequence of $d(\omega)$ is relativity compact in $\Sigma^{}(\mathbb{R}^{N})$. That is, if a sequence  $\left\{ u_{n}\right\}\subseteq \Sigma^{}(\mathbb{R}^{N})$ is such that $I_{\omega}(u_{n})=0$ and $S_{\omega}(u_{n})\rightarrow d(\omega)$ as $n$ goes to $+\infty$, then 
up to a subsequence there exist $\varphi\in\Sigma^{}(\mathbb{R}^{N}) $ satisfying  $S_{\omega}(\varphi)=d(\omega)$ and $u_{n}\rightarrow \varphi$ in $\Sigma^{}(\mathbb{R}^{N})$.\\
{\rm (ii)} The set of ground states is given by $\mathcal{G}_{\omega}=\left\{e^{i\theta}\phi_{\omega}: \, \theta\in\mathbb{R}\right\}$, where
\begin{equation}\label{IGS1}
\phi_{\omega}(x):=e^{\frac{\omega+\gamma N}{2}}e^{-\frac{\gamma}{2}\left|x\right|^{2}}, \quad x\in {\mathbb{R}}^{N}.
\end{equation}
\end{theorem}

\begin{remark}\rm
	\label{RM}
Assume that the infimum of $d(\omega)$ is achieved by $u$.  Then, there exist a Lagrange multiplier  $\Lambda\in \mathbb{R}$ such that $S^{\prime}_{\omega}(u)=\Lambda I^{\prime}_{\omega}(u)$, which implies that $\left\langle S^{\prime}_{\omega}( u),u\right\rangle=\Lambda\left\langle  I^{\prime}_{\omega}(u),u\right\rangle$. Hence    
$$
\text{$\left\langle S^{\prime}_{\omega}( u),u\right\rangle =I_{\omega}(u)=0$
	\quad and \quad $\left\langle  I^{\prime}_{\omega}(u),u\right\rangle= -2\left\|u_{}\right\|_{L^{2}}^{2}<0$}, 
$$
implies $\Lambda=0$. Therefore,  $u$ is a weak solution to equation \eqref{EP}. On the other hand, for any $v\in \Sigma^{}(\mathbb{R}^{N})\setminus  \left\{0 \right\}$ satisfying $S^{\prime}_{\omega}(v)=0$, it follows that 
$I_{\omega}(v)=0$. Thus, by definition of $\mathcal{G}_{\omega}$, we get that $u$ has  minimal action among all solutions of \eqref{EP}.
\end{remark}

We now discuss the notion of  stability of standing waves.  The basic symmetry associated to equation \eqref{0NL} is the {\em phase-invariance} (while the translation invariance  does {\em not} hold due to the harmonic potential); taking this fact into account,  it is reasonable
to define orbital stability as follows:

\begin{definition}\label{2D111}
We say that  a standing wave solution $u(x,t)=e^{i\omega t}\varphi(x)$ of \eqref{0NL} is orbitally stable in $\Sigma^{}(\mathbb{R}^{N})$ if for any  $\epsilon>0$ there exist $\eta>0$  with the following property: if $u_{0}\in \Sigma^{}(\mathbb{R}^{N})$ satisfies $\left\|u_{0}-\varphi \right\|_{\Sigma^{}}<\eta$, then the solution $u(t)$ of  \eqref{0NL}  with $u(0)=u_{0}$ exist for all $t\in \mathbb R$ and satisfies 
\begin{equation*}
{\rm\sup\limits_{t\in \mathbb R}} {\rm\inf\limits_{\theta\in \mathbb{R}}} \|u(t)-e^{i\theta}\varphi \|_{\Sigma}<\epsilon.
\end{equation*}
Otherwise, the standing wave $e^{i\omega t}\varphi(x)$ is said to be  unstable in $\Sigma^{}(\mathbb{R}^{N})$.
\end{definition}

Our second  result shows that,  in terms of the Cazenave and Lions' argument,  the ground states are orbitally stable.
\begin{theorem}[Orbital stability] 
	\label{2ESSW}
For any $\omega \in \mathbb{R}$ and $N\geq1$, the standing wave $e^{i\omega t}\phi_{\omega}(x)$ is orbitally stable in  
$\Sigma(\mathbb{R}^{N})$.
\end{theorem}

The paper is organized in the following way: in Section \ref{S:1}, we show that the energy functional $E$ is of class $C^{1}$ on $\Sigma(\mathbb{R}^{N})$. Moreover, we recall several known results and introduce several notations. In Section \ref{S:2}, we give an idea of the proof of  Proposition \ref{PCS}.  In Section \ref{S:3} we prove, by variational techniques, the existence of a minimizer for $d(\omega)$ (Theorem \ref{ESSW}), while in Section  \ref{S:4} the proof of Theorem \ref{2ESSW} is completed.

\bigskip
\noindent
{\bf Notation.} The space $L^{2}(\mathbb{R}^{N},\mathbb{C})$, denoted by $L^{2}(\mathbb{R}^{N})$ for shorthand, is equipped with the
norm $\|\cdot\|_{L^{2}}$. Moreover $2^{\ast}$ is defined by  $2^{\ast}=2N/(N-2)$ if $N\geq3$, and $2^{\ast}=+\infty$ if $N=1,2$. Finally, 
$\left\langle \cdot , \cdot \right\rangle$ is the duality pairing between $X^{\prime}$ and $X$, where $X$ is a Banach space and $X^{\prime}$ is its dual.

\section{Preliminary  lemmas}  
\label{S:1}
In this section we recall several known results, almost all are proved in the paper \cite{CL}. Moreover, we show that the energy functional $E$ is of class $C^{1}$ on $\Sigma(\mathbb{R}^{N})$.

\begin{proposition} \label{DFFE}
The energy functional $E$ defined by \eqref{EEEE} is of class $C^{1}$ and  for $u\in \Sigma(\mathbb{R}^{N})$ the  Fr\'echet derivative of $E$ in $u$ exists and it is given by  
\begin{equation*}
E^{\prime}(u)=-\Delta u+\gamma(\gamma-1)|x|^{2}u-u\, \mathrm{Log}\left|u\right|^{2}-u. 
\end{equation*}
\end{proposition}

\noindent
Before giving the proof of Proposition \ref{DFFE}, we fix some definitions that will be useful in the sequel. Define 
\begin{equation*}
F(z):=\left|z\right|^{2}\mbox{Log}\left|z\right|^{2},\quad  \text{for every  $z\in\mathbb{C}$},
\end{equation*}
and as in \cite{CL},  we introduce the functions  $A$, $B$ on $\left[0, \infty\right)$  by 
\begin{equation}\label{IFD}
A(s):=
\begin{cases}
-s^{2}\,\mbox{Log}(s^{2}), &\text{if $0\leq s\leq e^{-3}$;}\\
3s^{2}+4e^{-3}s^{}-e^{-6}, &\text{if $ s\geq e^{-3}$;}
\end{cases}
\quad  B(s):=F(s)+A(s).
\end{equation}
We will need the following  functions $a$ and $b$ given by
\begin{equation}\label{abapex}
a(z):=\frac{z}{|z|^{2}}\,A(\left|z\right|)\,\, \text{ and  }\,\, b(z):=\frac{z}{|z|^{2}}\,B(\left|z\right|),\quad \text{  for $z\in \mathbb{C}$, $z\neq 0$}.
\end{equation}
Noticing that for any $z\in\mathbb{C}$, $b(z)-a(z)=z\,\mathrm{Log}\left|z\right|^{2}$. In addition, we note that $A$ is a   nonnegative  convex and increasing function, and $A\in C^{1}\left([0,+\infty)\right)\cap C^{2}\left((0,+\infty)\right)$. We define the following Orlicz space $L^{A}(\mathbb{R}^{N})$ corresponding to $A$,
\begin{equation*}
L^{A}(\mathbb{R}^{N}):=\left\{u\in L^{1}_{\text{loc}}(\mathbb{R}^{N}) : A(\left|u\right|)\in L^{1}_{}(\mathbb{R}^{N})\right\}, 
\end{equation*}
equipped with the norm 
\begin{equation*}
\left\|u\right\|_{L^{A}}:={\inf}\left\{k>0: \int_{\mathbb{R}^{N}}A\left(k^{-1}{\left|u(x)\right|}\right)dx\leq 1 \right\}.
\end{equation*}
Here  $L^{1}_{\text{loc}}(\mathbb{R}^{N})$ is the space of all locally Lebesgue integrable functions. In  \cite[Lemma 2.1]{CL}, the author proved that  $\left(L^{A}(\mathbb{R}^{N}),\|\cdot\|_{L^{A}} \right)$ is a separable reflexive Banach space. 
Below we describe some properties of  $L^{A}(\mathbb{R}^{N})$. See \cite[Lemma 2.1]{CL} for more details.

\begin{proposition} \label{orlicz}
Assume that $\left\{u_{{m}}\right\}$ is a sequence  in  $L^{A}(\mathbb{R}^{N})$. Then the following facts hold:\\
	{\it i)} If  $u_{{m}}\rightarrow u$ in $L^{A}(\mathbb{R}^{N})$, then $A(\left|u_{{m}}\right|)\rightarrow A(\left|u\right|)$ in $L^{1}(\mathbb{R}^{N})$ as   $m\rightarrow \infty$.\\
	{\it ii)} Let  $u\in L^{A}(\mathbb{R}^{N})$. If  $u_{m}\rightarrow u$ $a.e.$ in $\mathbb{R}^{N}$ and if 
	\begin{equation*}
	\lim_{m \to \infty}\int_{\mathbb{R}^{N}}A\left(\left|u_{m}(x)\right|\right)dx=\int_{\mathbb{R}^{N}}A\left(\left|u(x)\right|\right)dx,
	\end{equation*}
	then $u_{{m}}\rightarrow u$ in $L^{A}(\mathbb{R}^{N})$ as $m\rightarrow \infty$.\\
	{\it iii)} For any function $u$ in $L^{A}(\mathbb{R}^{N})$,  we have the following relationship
\begin{equation}\label{DA1}
	{\rm min} \left\{\left\|u\right\|_{L^{A}},\left\|u\right\|^{2}_{L^{A}}\right\}\leq  \int_{\mathbb{R}^{N}}A\left(\left|u(x)\right|\right)dx\leq {\rm max} \left\{\left\|u\right\|_{L^{A}},\left\|u\right\|^{2}_{L^{A}}\right\}.
	\end{equation}
\end{proposition}

\begin{remark}\rm
A simple calculation shows  that for all  $\epsilon>0$, there exists $C_{\epsilon}>0$ with 
$$
|B(z)-B(w)|\leq C_{\epsilon}(|z|^{1+\epsilon}+|w|^{1+\epsilon})|z-w|, 
\quad \text{for every $z,w\in \mathbb{C}$}.  
$$
hence, integrating  on $\mathbb{R}^{N}$ with $\epsilon=(2^{*}_{}-2)/2$ and applying H\"{o}lder's inequality and Sobolev's Inequalities we see that 
\begin{equation}\label{DB}
\int_{\mathbb{R}^{N}}\left|B(\left|u\right|)- B(\left|v\right|)\right|dx\leq C\left(1+ \left\|u\right\|^{2}_{H^{1}(\mathbb{R}^{N})}+ \left\|v\right\|^{2}_{H^{1}(\mathbb{R}^{N})} \right)\left\|u-v\right\|_{{L^{2}}}
\end{equation}
for all $u$, $v\in H^{1}(\mathbb{R}^{N})$. 
\end{remark}

\begin{lemma}[Injections] 
	\label{POPO}
Let $N\geq 1$. Then the following assertions hold.\\
{\rm (i)} The embedding $\Sigma(\mathbb{R}^{N})\hookrightarrow L^{q}(\mathbb{R}^{N})$ is compact, where $2\leq q< 2^{\ast}$.\\
{\rm (ii)} The inclusion map $\Sigma(\mathbb{R}^{N})\hookrightarrow L^{2-\delta}(\mathbb{R}^{N})$  is continuous, where $\delta={1}/{N}$.\\
{\rm (iii)} The inclusion map $\Sigma(\mathbb{R}^{N})\hookrightarrow L^{A}(\mathbb{R}^{N})$  is continuous.
\end{lemma} 
\begin{proof} Item (i) is proved in  \cite[Lemma 3.1]{JZ2000}. Let $u\in \Sigma(\mathbb{R}^{N})$. By  H\"{o}lder's inequality with conjugate exponents $2N/(2N-1)$, $2N$ we see that  
\begin{equation*}
\int_{\mathbb{R}^{N}}|u(x)|^{2-\frac{1}{N}}dx\leq \left(\int_{\mathbb{R}^{N}}\frac{1}{(1+|x|^{2})^{\alpha}}dx\right)^{\frac{1}{2N}}\left(\int_{\mathbb{R}^{N}}(1+|x|^{2})|u(x)|^{2}dx\right)^{\frac{2N-1}{2N}},
\end{equation*}
where $\alpha=2N-1$. Since $\alpha>N/2$, we have that there exists a constant $C>0$ depending only on $N$ such that $\|u\|_{L^{2-{1}/{N}}(\mathbb{R}^{N})}\leq C\|u\|_{\Sigma}$, which  completes the proof of Item (ii). Concerning (iii), it follows form \eqref{IFD}  that  for every  $N\in \mathbb{N}$, there exist  $C>0$ depending only on $N$ such that
\begin{equation*}
A(|z|)\leq C(|z|^{2+\frac{1}{N}}+|z|^{2-\frac{1}{N}})+B(|z|) \quad \text{for any $z\in \mathbb{C}$}.
\end{equation*}
Notice that $2 <2+{1}/{N}< 2^{\ast}$. Thus from Item (i), Item (ii) and \eqref{DB} we have that if $u_{n}\rightarrow u$ as $n$ goes to $+\infty$ in $\Sigma(\mathbb{R}^{N})$, then $A(|u_{n}-u|)\rightarrow 0$ as $n$ goes to $+\infty$ in $L^{1}(\mathbb{R}^{N})$. This implies by \eqref{DA1} that $u_{n}\rightarrow u$ as $n$ goes to $+\infty$ in $L^{A}(\mathbb{R}^{N})$. This concludes the proof.
\end{proof}

For a proof of following result, see \cite[Lemma 2.5 and Lemma 2.6]{CL}.
\begin{lemma} \label{POPty}
Let $N\geq 1$ and  consider the functions $a$ and $b$ defined by \eqref{abapex}.  Then the following is true.\\
{\rm (i)} The operator $u\rightarrow a(u)$ is continuous from  $L^{A}(\mathbb{R}^{N})$  into $L^{A^{\prime}}(\mathbb{R}^{N})$. Moreover, the image under a of  every bounded subset of $L^{A}(\mathbb{R}^{N})$ is a bounded subset of $L^{A^{\prime}}(\mathbb{R}^{N})$. \\
{\rm (ii)} The operator $u\rightarrow b(u)$ is continuous from  $H^{1}(\mathbb{R}^{N})$  into $H^{-1}(\mathbb{R}^{N})$. Moreover, the image under b of every bounded subset of $H^{1}(\mathbb{R}^{N})$ is a bounded subset of $H^{-1}(\mathbb{R}^{N})$.\\
\end{lemma} 

\begin{lemma} \label{APEX23}
The operator 
$$
L: u\rightarrow -\Delta u+\gamma(\gamma-1)|x|^{2}u-u\,  \mathrm{Log}\left|u\right|^{2}
$$ 
is continuous from  $\Sigma(\mathbb{R}^{N})$  to $\Sigma^{\prime}(\mathbb{R}^{N})$. The image under $L$  of every bounded subset of $\Sigma(\mathbb{R}^{N})$ is a bounded subset of $\Sigma^{\prime}(\mathbb{R}^{N})$.
\end{lemma}
\begin{proof}
It is clear $-\Delta +\gamma(\gamma-1)|x|^{2}$ is continuous from  $\Sigma(\mathbb{R}^{N})$ to $\Sigma^{\prime}(\mathbb{R}^{N})$. 
Hence, we need to prove the continuity of the nonlinearity part of $L$.  Indeed, by Lemma \ref{POPty}(i), $u\rightarrow a(u)$ is continuous  from $L^{A}({\mathbb{R}}^{N})$ to $L^{A^{\prime}}({\mathbb{R}}^{N})$,   which implies by Lemma \ref{POPO}(iii), from $\Sigma(\mathbb{R}^{N})$ to $\Sigma^{\prime}(\mathbb{R}^{N})$. Finally, applying  Lemma \ref{POPty}(ii), we see that the operator $u\rightarrow a(u)-b(u)=-u\,\mathrm{Log}\left|u\right|^{2}$  is continuous  from $\Sigma(\mathbb{R}^{N})$ to $\Sigma^{\prime}(\mathbb{R}^{N})$. This completes of proof Lemma \ref{APEX23}.
\end{proof}

\begin{proof}[\bf {Proof of Proposition \ref{DFFE}}] 
We first show that $E$  is continuous on $\Sigma(\mathbb{R}^{N})$. Note that $E$ can be rewritten in the following form
\begin{equation}\label{CCC}
E(u)=\frac{1}{2}\|\nabla u\|_{L^{2}}^{2}+\frac{\gamma(\gamma-1)}{2}\|xu\|_{L^{2}}^{2} +\frac{1}{2}\int_{\mathbb{R}^{N}}A(\left|u_{}\right|)dx-\frac{1}{2}\int_{\mathbb{R}^{N}}B(\left|u_{}\right|)dx.
\end{equation}
The first and second term in the right-hand side of \eqref{CCC} are continuous on $\Sigma(\mathbb{R}^{N})$. Hence, we need to prove the continuity of the nonlinearity part of $E$. Combining  Proposition \ref{orlicz}(i) and Lemma \ref{POPO}(iii) we obtain that  the third term is continuous on $\Sigma(\mathbb{R}^{N})$. Moreover, by \eqref{DB} we have that the fourth term  in the right-hand side of \eqref{CCC} is continuous on $\Sigma(\mathbb{R})$, which implies that $E\in C(\Sigma(\mathbb{R}^{N}),\mathbb{R})$. Next it is easily seen  that, for $u,v\in \Sigma(\mathbb{R}^{N})$, $t\in (-1,1)$ (see \cite[Proposition 2.7]{CL}),
\begin{equation*}
\lim_{t\rightarrow 0} \frac{E(u+tv)-E(u)}{t}=\bigl\langle (-\Delta u +\gamma(\gamma-1)|x|^{2}u-u\, \mbox{Log}\left|u\right|^{2}-u,v\bigl\rangle.
\end{equation*}
Thus, $E$ is G\^ateaux differentiable.  By virtude of Lemma \ref{APEX23} we conclude that $E$ is  Fr\'echet differentiable.
\end{proof}

\section{The Cauchy problem}
\label{S:2}

In this section we sketch the proof of the global well-posedness of \eqref{0NL} for any $\gamma > 1$  as stated in Proposition \ref{PCS}. A similar technique was applied by Cazenave \cite[Theorem 9.3.4]{CB} in the case of the NLS equation \eqref{0NL} without the term $V(x)$. We first construct a sequence of global weak  solutions of a regularized Cauchy problem in $C(\mathbb{R},\Sigma(\mathbb{R}^{N}))$ which converges to a weak solution of the equation \eqref{0NL}. This produces a weak solution. Then, applying some properties of the logarithmic nonlinearity  we show the uniqueness of the weak solution of equation \eqref{0NL}.

Before outlining the main ideas of the proof of Proposition \ref{PCS}, we fix some definitions that will be useful in the sequels. For $z\in \mathbb{C}$ and $m\in \mathbb{N}$, we introduce the functions $a_{m}$ and $b_{m}$ by 
\begin{equation*}
a_{m}(z)=
z\tilde a_m(|z|),\quad\, 
\tilde a_m(s):=
\begin{cases}
\frac{A(s)}{s^{2}}, &\text{if $s\geq \frac{1}{m}$,}\\
m^2 A(\frac{1}{m}) , &\text{if $0\leq s\leq \frac{1}{m}$,}
\end{cases}
\end{equation*}
\begin{equation*}
b_{m}(z)=
z\tilde b_m(|z|),\quad\, 
\tilde b_m(s):=
\begin{cases}
\frac{B(s)}{s^{2}} , &\text{if $0\leq s\leq {m}$,}\\
\frac{B(m)}{m^2}, &\text{if $s\geq {m}$.}
\end{cases}
\end{equation*}
where $A$ and $B$ were defined in \eqref{abapex}. We will need the following family of  nonlinearities given by $g_{m}(z)=b_{m}(z)-a_{m}(z)$ for any fixed $m\in \mathbb{N}$ and for every $z\in \mathbb{C}$.

Now we need to construct an appropriate sequence of weak solutions of the following regularized Cauchy problem
\begin{equation}\label{AHAX}
i\partial_{t}u^{m}+\Delta u^{m}-\gamma(\gamma-1)|x|^{2}u^{m} +g_{m}(u^{m})=0, \quad m\in \mathbb{N}.\\
\end{equation}

\begin{proposition} \label{APCS}
For every  $u_{0}\in \Sigma(\mathbb{R}^{N})$, there is a unique global solution  $u^{m}\in C(\mathbb{R},\Sigma(\mathbb{R}^{N}))\cap C^{1}(\mathbb{R}, \Sigma^{\prime}(\mathbb{R}^{N}))$ of problem \eqref{AHAX}  with $u(0)=u_{0}$. Furthermore, the mass and total energy associated with \eqref{AHAX} are conserved in time, namely
\begin{equation}\label{ACLW}
{E}_{m}(u^{m}(t))={E}_{m}(u_{0})\quad  and \quad \left\|u^{m}(t)\right\|^{2}_{L^{2}}=\left\|u_{0}\right\|^{2}_{L^{2}}\quad  \text{for all $t\in \mathbb{R}$},
\end{equation}
where 
\begin{equation*}
E_{m}(u)=\frac{1}{2}\left\|\nabla u^{}\right\|^{2}_{L^{2}}+\frac{\gamma(\gamma-1)}{2}\left\|x u^{}\right\|^{2}_{L^{2}}   +\int_{\mathbb{R}^N}\Phi_{m}(|u|)dx-\int_{\mathbb{R}^N}\Psi_{m}(|u|)dx,
\end{equation*}
and the functions $t\mapsto \Phi_{m}(t)$ and $t\mapsto \Psi_{m}(t)$  are defined by
\begin{equation*}
\Phi_{m}(t):=\int^{t}_{0}s\tilde a_{m}(s)ds \quad \text{and} \quad \Psi_{m}(t):=\int^{t}_{0}s\tilde b_{m}(s)ds.
\end{equation*}
\end{proposition}
\begin{proof}
Since $g_{m}$ is globally Lipschitz continuous $\mathbb{C}\rightarrow \mathbb{C}$, the global well-posedness
follows from Strichartz inequalities and a fixed point argument; see e.g.\ \cite[Theorem 9.2.6 and Remark 9.2.8]{CB}.
\end{proof}

In the following we will make use of the following lemma.
\begin{lemma}\label{3ACS} 
Assume that $\left\{u^{{m}}\right\}_{m\in\mathbb{N}}$ is a sequence bounded in  $L^{\infty}(\mathbb{R},\Sigma(\mathbb{R}^{N}))$ and in $W^{1,\infty}(\mathbb{R},\Sigma^{\prime}(\mathbb{R}^{N}))$. Then  there exist a subsequence, which we still denote by $\left\{u^{{m}}\right\}_{m\in\mathbb{N}}$, and there exist  a function $u\in L^{\infty}(\mathbb{R},\Sigma(\mathbb{R}^{N}))\cap W^{1,\infty}(\mathbb{R},\Sigma^{\prime}(\mathbb{R}^{N}))$ such that the following conclusions are valid.\\
{\rm (i)}  $u^{m}(t)\rightharpoonup u^{}(t)$  in  $\Sigma(\mathbb{R}^{N})$  as $m\rightarrow \infty$ for all $t\in \mathbb{R}$.\\
{\rm (ii)}	For any $t\in\mathbb{R}$ there is a subsequence  $m_{j}$ with $u_{}^{m_{j}}(x,t)\rightarrow u_{}^{}(x,t)$   as $j\rightarrow \infty$, for a.e. $x\in \mathbb{R}^{N}$.\\
{\rm (iii)}	$u_{}^{m_{}}(x,t)\rightarrow u_{}^{}(x,t)$ as $m\rightarrow \infty$,  for a.e.  $(x,t)\in \mathbb{R}^{N}\times\mathbb{R}$. 
\end{lemma}
\begin{proof}
The proof  follows a similar argument used in  \cite[Lemma 9.3.6]{CB}, and we omit the details.
\end{proof}

\begin{proof}[\bf {Proof of Proposition \ref{PCS}}] Here, for simplicity, we assume that $\gamma(\gamma-1)=1$. We proceed by approximating
the equation as follows (see \cite[Theorem 9.3.4]{CB}): taking into account Lemma \ref{APCS}, we get that exists  a unique solution $u^{m}\in C(\mathbb{R},\Sigma(\mathbb{R}^{N}))\cap C^{1}(\mathbb{R}, \Sigma^{\prime}(\mathbb{R}^{N}))$ of the regularized NLS equation \eqref{AHAX} with $u(0)=u_{0}$. In turn, by combining the conservation of energy and charge \eqref{ACLW}  we obtain that the sequence of approximating solutions $u^{m}$ is bounded in $L^{\infty}(\mathbb{R}, \Sigma(\mathbb{R}^{N}))$  (see Step 2 of \cite[Theorem 9.3.4]{CB} for example). On the other hand, notice that the following inequality can be easily shown
\begin{equation*}
|g_{m}(z)|^{2}\leq C(|z|^{2+\frac{1}{N}}+|z|^{2-\frac{1}{N}}) \quad \text{for all $z\in \mathbb{C}$ and all $m\in \mathbb{N}$}.
\end{equation*}
Hence, by Lemma \ref{POPO}, we get that $g_{m}(u^{m})$ is bounded in $L^{\infty}(\mathbb{R}, \Sigma^{\prime}(\mathbb{R}^{N}))$, which implies by \eqref{AHAX}  that the sequence $\partial_{t}u^{m}$ is bounded in $L^{\infty}(\mathbb{R}, \Sigma^{\prime}(\mathbb{R}^{N}))$. 
We can now conclude that $\left\{u^{{m}}\right\}_{m\in\mathbb{N}}$ satisfies the conditions of  Lemma \ref{3ACS}. Let $u$ be the limit of $u^{m}$.

We claim that the limiting function $u\in L^{\infty}(\mathbb{R}, \Sigma(\mathbb{R}^{N}))$ is a weak solution of the NLS \eqref{0NL}. Indeed, 
it follows from  property (i) of Lemma \ref{3ACS} that  $u(0)=u_{0}$. Furthermore, by \eqref{AHAX}, for any test  function $\psi\in C^{\infty}_{0}({\mathbb{R}^{N}})$  and $\phi\in C^{\infty}_{0}({\mathbb{R}})$ we get
\begin{equation}\label{3DPL}
\int_{\mathbb{R}}\left[-\left\langle i\, u^{m}, \psi\right\rangle \phi^{\prime}(t)+ \left\langle u^{m}, \psi\right\rangle_{\Sigma} \phi^{}(t)\right]\,dt+\int_{\mathbb{R}}\int_{\mathbb{R}^{N}}g_{m}(u^{m})\psi\phi\,dx\,dt=0.
\end{equation}
Taking into account that $g_{m}(z)\rightarrow z\, \mbox{Log}\left|z\right|^{2}$ pointwise in $z\in \mathbb{C}$ as $m\rightarrow+\infty$, via properties (i)-(iii) of Lemma \ref{3ACS} we get the following integral equation (see Step 3 of \cite[Theorem 9.3.4]{CB})
 \begin{equation}\label{ert}
\int_{\mathbb{R}}\left[-\left\langle i\, u^{}, \psi\right\rangle \phi^{\prime}(t)+ \left\langle u^{}, \psi\right\rangle_{\Sigma} \phi^{}(t)\right]\,dt+\int_{\mathbb{R}}\int_{\mathbb{R}^{N}}u_{}\, \mbox{Log}\left|u_{}\right|^{2}\psi\phi\,dx\,dt=0.
\end{equation}

We are now ready to conclude the proof of proposition. Since $u\in L^{\infty}(\mathbb{R},\Sigma(\mathbb{R}^{N}))$, in view of Lemma \ref{APEX23} and \eqref{ert} we see that $u_{t}\in L^{\infty}(\mathbb{R},\Sigma^{\prime}(\mathbb{R}^{N}))$ and $u$ is a weak solution of problem \eqref{0NL}. Next, to prove the uniqueness of the weak solution and the conservation of charge and energy one can follow the argument of \cite[Theorem 9.3.4]{CB}. The proof is now concluded.
\end{proof}

\section{Variational analysis}
\label{S:3}

This section is devoted to the proof of Theorem \ref{ESSW}.  We begin with the logarithmic Sobolev inequality. See \cite[Theorem 8.14]{ELL}.
\begin{lemma} \label{L1}
Let $\alpha>0$ and assume that $f\in H^{1}(\mathbb{R}^{N})\setminus\{0\}$.  Then
\begin{equation}\label{SII}
\int_{\mathbb{R}^{N}}\left|f(x)\right|^{2}\mathrm{Log}\left|f(x)\right|^{2}dx\leq \frac{\alpha^{2}}{\pi} \|\nabla f \|^{2}_{L^{2}}+\left(\mathrm{Log} \|f \|^{2}_{L^{2}}-N\left(1+\mathrm{Log}\,\alpha\right)\right)\|f \|^{2}_{L^{2}}.
\end{equation}
Furthermore, there is equality if and only if the function $f$ is, up to translation, a multiple of  $e^{\left\{-\pi \left|x\right|^{2}/2 \alpha^{2}\right\}}$.
\end{lemma}

\begin{lemma}[Ground energy]
	\label{L2}
Let $\omega\in \mathbb{R}$ and $\gamma>1$. Then, the quantity $d(\omega)$ is given by
\begin{equation}\label{EA}
d(\omega)=\frac{1}{2}\|\phi_{\omega}\|_{L^{2}}^{2}=\frac{1}{2} {{\pi}^{\frac{N}{2}}{\gamma}^{-\frac{N}{2}}}e^{\omega+\gamma N},
\end{equation}
where $\phi_{\omega}$ is defined by \eqref{IGS1}.
\end{lemma}
\begin{proof} We observe for further usage that $\|\phi_{\omega}\|_{L^{2}}^{2}={{\pi}^{\frac{N}{2}}{\gamma}^{-\frac{N}{2}}}e^{\omega+\gamma N}$ for every $\omega\in \mathbb{R}$. We first prove $2 d(\omega)\leq \|\phi_{\omega}\|_{L^{2}}^{2}$.  By direct computations, we obtain that $I_{\omega}(\phi_{\omega})=0$, which implies that, by the definition of $d(\omega)$,  $2 d(\omega)\leq \|\phi_{\omega}\|_{L^{2}}^{2}$. On the other hand, it is easily seen that
\begin{equation}\label{INF1}
\inf\left\{\|\nabla u\|_{L^{2}}^{2}+\gamma^{2}\|xu\|_{L^{2}}^{2}:u\in \Sigma(\mathbb{R}^{N}), \| u\|_{L^{2}}^{2}=1 \right\}=\gamma N.
\end{equation}
In particular, multiplying \eqref{INF1} by $\gamma^{-1}(\gamma-1)$ we get
\begin{equation}\label{IN}
 (\gamma-1)N\| u\|_{L^{2}}^{2}\leq\gamma^{-1}(\gamma-1)\|\nabla u\|_{L^{2}}^{2}+\gamma^{}(\gamma-1)\|xu\|_{L^{2}}^{2}.
\end{equation}
Now, let $u\in \Sigma(\mathbb{R}^{N})\setminus \left\{0 \right\}$ be such that  $I_{\omega}(u)=0$. By virtue of  the logarithmic Sobolev inequality with $\alpha^{2}={\pi}/{\gamma}$ and inequality \eqref{IN} we get
\begin{equation*}
\left(\omega+\gamma N+N\mbox{Log}(\sqrt{{\pi}/{\gamma}})\right)\left\|u\right\|^{2}_{L^{2}}\leq \left(\mbox{Log}\left\|u\right\|^{2}_{L^{2}}\right)\left\|u\right\|^{2}_{L^{2}},
\end{equation*}
which implies that $\left\|u\right\|^{2}_{L^{2}}\geq \|\phi_{\omega}\|_{L^{2}}^{2}$. Then, in view of the definition of $d(\omega)$, it follows  that $2 d(\omega)\geq \|\phi_{\omega}\|_{L^{2}}^{2}$. This conclude the proof.
\end{proof}

Next we give a useful lemma.

\begin{lemma} \label{L4}
Assume that $\left\{u_{n}\right\}$ is a bounded sequence in $\Sigma(\mathbb{R}^{N})$ satisfying as $n\rightarrow \infty$, $u_{n}\rightarrow u$ a.e. in $\mathbb{R}^{N}$. Then $u\in \Sigma(\mathbb{R}^{N})$ and 
\begin{equation*}
\lim_{n\rightarrow \infty}\int_{\mathbb{R}^{N}}\left\{\left|u_{n}\right|^{2}\mathrm{Log}\left|u_{n}\right|^{2}-\left|u_{n}-u\right|^{2}\mathrm{Log}\left|u_{n}-u\right|^{2}\right\}dx=\int_{\mathbb{R}^{N}}\left|u\right|^{2}\mathrm{Log}\left|u\right|^{2}dx.
\end{equation*}
\end{lemma}
\begin{proof}
Taking  into account that $\Sigma(\mathbb{R}^{N})\hookrightarrow L^{A}(\mathbb{R}^{N})$, the assertion follows by \cite[Lemma 2.3]{AHA1} (see also \cite{LBL}).
\end{proof}

Now we give the proof of Theorem \ref{ESSW}.

\begin{proof}[{\bf{Proof of Theorem \ref{ESSW}}}]  Let $\left\{ u_{n}\right\} \subseteq \Sigma(\mathbb{R}^{N})$ be a minimizing sequence for $d(\omega)$, namely, $I_{\omega}(u_{n})=0$ for all $n$, and $S_{\omega}(u_{n})\rightarrow d(\omega)$ as $n\rightarrow \infty$. Notice that sequence $\left\{ u_{n}\right\}$ is bounded in  $\Sigma(\mathbb{R}^{N})$. In fact, it is clear that the sequence $\|u_{n}\|^{2}_{L^{2}}$ is bounded. Furthermore, by virtue of the logarithmic Sobolev inequality \eqref{SII} and recalling that $I_{\omega}(u_{n})=0$, we obtain
\begin{equation*}
\left(1-\frac{\alpha^{2}}{\pi}\right)\left\|\nabla u_{n}\right\|^{2}_{L^{2}}+{\gamma(\gamma-1)}\|xu_{n}\|_{L^{2}}^{2}\leq \mbox{Log}\left[\left(\frac{e^{-\left(\omega+N\right)}}{\alpha^{N}}\right)\left\|u_{n}\right\|^{2}_{L^{2}}\right]\left\|u_{n}\right\|^{2}_{L^{2}}.
\end{equation*}
Now by taking sufficiently small positive $\alpha>0$ enables us to conclude that all minimizing sequences are bounded in $\Sigma(\mathbb{R}^{N})$. This implies that there exists some function $\varphi \in \Sigma(\mathbb{R}^{N})$ such that, up to a subsequence, $u_{n}\rightharpoonup \varphi$ weakly in $\Sigma(\mathbb{R}^{N})$ and this implies, by virtue of Lemma \ref{POPO}(i) that as $n$ goes to $+\infty$, $u_{n}\rightarrow\varphi$ in $L^{q}(\mathbb{R}^{N})$ for $2\leq q < 2^{\ast}$. In particular, we get $\|\varphi\|^{2}_{L^{2}}=2d(\omega)$. 

Now, let us prove that $I_{\omega}(\varphi)=0$.   Assume by contradiction that $I_{\omega}(\varphi)<0$. Notice that by simple  computations, we can see that there is $0<\lambda<1$ such that $I_{\omega}(\lambda \varphi)=0$. In view of definition of $d(\omega)$, we get
\begin{equation*}
d(\omega)\leq S_{\omega}(\lambda\varphi)=\frac{1}{2}\left\|\lambda \varphi\right\|^{2}_{L^{2}}<\frac{1}{2}\left\|\varphi\right\|^{2}_{L^{2}}=
d(\omega),
\end{equation*}
a contradiction. On the other hand, assume that $I_{\omega}(\varphi)>0$. Since  $u_{n}\rightharpoonup \varphi$  in $\Sigma(\mathbb{R}^{N})$, it follows that
\begin{align}
&\left\|\nabla u_{n}\right\|^{2}_{L^{2}}-\left\|\nabla u_{n}-\nabla\varphi\right\|^{2}_{L^{2}}-\left\|\nabla \varphi \right\|^{2}_{L^{2}}\rightarrow0 \label{2C11}\\
& \left\|xu_{n}\right\|^{2}_{L^{2}}-\left\|xu_{n}-x\varphi\right\|^{2}_{L^{2}}-\left\|x\varphi\right\|^{2}_{L^{2}}\rightarrow0,  \
\label{2C12} 
\end{align}
as $n\rightarrow\infty$. By combining \eqref{2C11} with \eqref{2C12} and Lemma \ref{L4} leads to 
\begin{equation*}
\lim_{n\rightarrow \infty}I_{\omega}(u_{n}-\varphi)=\lim_{n\rightarrow \infty}I_{\omega}(u_{n})-I_{\omega}(\varphi)=-I_{\omega}(\varphi),
\end{equation*}
which combined with  $I_{\omega}(\varphi)> 0$ implies  that $I_{\omega}(u_{n}-\varphi)<0$ for sufficiently large $n$. Then, by arguing as above,  we can prove that
\begin{equation*}
d(\omega)\leq\frac{1}{2} \lim_{n\rightarrow \infty}\left\|u_{n}-\varphi\right\|^{2}_{L^{2}}=d(\omega)-\frac{1}{2}\left\|\varphi\right\|^{2}_{L^{2}},
\end{equation*}
which is a contradiction. We get $I_{\omega}(\varphi)=0$, and this implies, by virtue of the definition of $d(\omega)$, that $\varphi\in \mathcal{G}_{\omega}$.

Next we prove that $u_{n}\rightarrow\varphi$ in $\Sigma(\mathbb{R}^{N})$. Notice that, on one hand, we have $u_{n}\rightarrow \varphi$  in $L^{2}(\mathbb{R}^{N})$. On the other hand, since the sequence $\left\{ u_{n}\right\}$ is bounded in $\Sigma(\mathbb{R}^{N})$, it follows by \eqref{DB} that
\begin{equation*}
 \lim_{n\rightarrow \infty}\int_{\mathbb{R}^{N}}B\left(\left|u_{n}(x)\right|\right)dx=\int_{\mathbb{R}^{N}}B\left(\left|\varphi(x)\right|\right)dx,\end{equation*}
which combined with $I_{\omega}(u_{n})=I_{\omega}(\varphi)=0$  for any $n\in \mathbb{N}$,  gives
\begin{multline}\label{2BX1}
\lim_{n\rightarrow \infty}\left[\|\nabla u_{n}\|^{2}_{L^{2}}+ \gamma(\gamma-1)\|x u_{n}\|^{2}_{L^{2}} +\int_{\mathbb{R}^{N}}A\left(\left|u_{n}\right|\right)dx\right]\\
=\|\nabla \varphi\|^{2}_{L^{2}}+ \gamma(\gamma-1)\|x \varphi\|^{2}_{L^{2}} +\int_{\mathbb{R}^{N}}A\left(\left|\varphi\right|\right)dx,
\end{multline}
and this implies, by virtue of \eqref{2BX1},  the weak lower semicontinuity  and Fatou's Lemma, that (see e.g. \cite[Lemma 12 in chapter V]{AH})
\begin{align}
& \lim_{n\rightarrow \infty}\|\nabla u_{n}\|^{2}_{L^{2}}=\|\nabla \varphi\|^{2}_{L^{2}}, \quad  \lim_{n\rightarrow \infty}\|x u_{n}\|^{2}_{L^{2}}=\|x \varphi\|^{2}_{L^{2}}.  \label{N1}
 \end{align}
Therefore, it follows from  \eqref{N1} that $u_{n}\rightarrow\varphi$  in $\Sigma(\mathbb{R}^{N})$.  This proves the first part of the statement of Theorem \ref{ESSW}.

Now we claim that $|\varphi|\in \mathcal{G}_{\omega}$ and $|\varphi|$ is necessarily  radially symmetric. Indeed, denoting by $\varphi^{\ast}$ the Schwarz symmetrization of $|\varphi|$, since $A, B \in C^{1}([0,+\infty))$ are increasing functions with $A(0)=B(0)=0$, it is follows from Layer cake representation \cite[Theorem 1.13]{ELL} and \eqref{IFD} that
\begin{equation*}
\int_{\mathbb{R}^{N}}\left|\varphi^{\ast}(x)\right|^{2}\mathrm{Log}\left|\varphi^{\ast}(x)\right|^{2}dx=\int_{\mathbb{R}^{N}}\left|\varphi(x)\right|^{2}\mathrm{Log}\left|\varphi(x)\right|^{2}dx.
\end{equation*}
Moreover, as it is readily checked,
\begin{equation*}
\int_{\mathbb{R}^{N}}|x|^{2}\left|\varphi^{\ast}(x)\right|^{2}dx<\int_{\mathbb{R}^{N}}|x|^{2}\left|\varphi^{}(x)\right|^{2}dx\quad \text{unless} \quad |\varphi|=\varphi^{\ast}\,\, a.e.
\end{equation*}
Thus, since we have that $\|\nabla \varphi^{\ast}\|^{2}_{L^{2}}\leq \|\nabla |\varphi|\|^{2}_{L^{2}}\leq\|\nabla \varphi\|^{2}_{L^{2}}$ and $\| \varphi^{\ast}\|^{2}_{L^{2}}= \| \varphi\|^{2}_{L^{2}}$, it follows that if $|\varphi|\neq \varphi^{\ast}$, then $I_{\omega}(\varphi^{\ast})<I_{\omega}(\varphi)=0$ with $\| \varphi^{\ast}\|^{2}_{L^{2}}= \| \varphi\|^{2}_{L^{2}}$, which is a contradiction because  $\varphi\in \mathcal{G}_{\omega}$. This contradiction finishes the proof of claim.

By virtue of Lemma \ref{L2} it follows that $\left\{e^{i\theta}\phi_{\omega}: \, \theta\in\mathbb{R}\right\}\subseteq \mathcal{G}_{\omega}$.  Next let us consider $\varphi\in \mathcal{G}_{\omega}$. Taking into account the definition of $d(\omega)$, $\left\|\varphi\right\|^{2}_{L^{2}}=2d(\omega)$ and  $I_{\omega}(\varphi)=0$. We claim that the function $\varphi$ satisfies the equality  in \eqref{SII} with $\alpha^{2}=\pi/\gamma$. Let us assume the contrary, i.e. suppose that we have the strict inequality in \eqref{SII} with $\alpha^{2}=\pi/\gamma$. 
Since $\varphi$ satisfies $I_{\omega}(\varphi)=0$, a direct computation yields  $\left\|\varphi\right\|^{2}_{L^{2}}>2d(\omega)$ (see proof of Lemma \ref{L2}), a contradiction. Thus,  in light of Lemma \ref{L1} we have  that there exist $r>0$, $y\in \mathbb{R}^{N}$ and $\theta_{0}\in \mathbb{R}^{}$  such that 
\begin{equation*}
\varphi(x)=r\, e^{i\theta_{0}}e^{-\frac{\gamma}{2} \left|x-y\right|^{2}}, \quad  x\in \mathbb{R}^{N}.
\end{equation*}
Since $|\varphi|$ is radial and $\left\|\varphi\right\|^{2}_{L^{2}}=2d(\omega)$, we conclude that $y=0$ and $r^{2}=e^{\omega +\gamma N}$. Hence,  $\varphi(x)= e^{i\theta_{0}}\,\phi_{\omega}(x)$ and the Theorem \ref{ESSW} is proved.
\end{proof}

\section{Stability of standing waves}
\label{S:4}

\noindent
\begin{proof}[{\bf{Proof of Theorem \ref{2ESSW}}}] 
We argue by contradiction. Suppose that $\phi_{\omega}$ is not stable in $\Sigma(\mathbb{R}^{N})$ under flow associated with problem \eqref{0NL}. Then there exist $\epsilon>0$, a sequence of initial data $(u_{n,0})_{n\in \mathbb{N}}$ in $\Sigma(\mathbb{R}^{N})$ such that for all $n\geq1$,
\begin{equation}\label{T21}
\left\|u_{n,0}-\phi_{\omega}\right\|_{\Sigma}<\frac{1}{n},
\end{equation}
and a sequence $(t_{n})_{n\in \mathbb{N}}$ such that
\begin{equation}\label{3C2}
\inf_{\theta\in\mathbb{R}} \|u_{n}(t_{n})-e^{i\theta}\phi_{\omega}\|_{\Sigma}\geq{\epsilon}, \quad \text{for any $n\in \mathbb{N}$,}
\end{equation}
where $u_{n}$ denotes the unique solution of problem \eqref{0NL} with initial data $u_{n,0}$. Now, setting $v_{n}(x)= u_{n}(x,t_{n})$ it follows  by \eqref{T21} and conservation laws
\begin{gather}
\left\|v_{n}\right\|^{2}_{L^{2}}=\left\|u_{n}(t_{n})\right\|^{2}_{L^{2}}=\left\|u_{n,0}\right\|^{2}_{L^{2}}\rightarrow \left\|\phi_{\omega}\right\|^{2}_{L^{2}}\label{CE1} \\
E(v_{n})=E(u_{n}(t_{n}))=E(u_{n,0})\rightarrow E(\phi_{\omega}).
\label{CE2}
\end{gather}
Consequently, by virtue of \eqref{CE1} and \eqref{CE2}, 
\begin{equation}\label{A12}
S_{\omega}(v_{n})\rightarrow S_{\omega}(\phi_{\omega})=d(\omega).
\end{equation}
Thus,  \eqref{CE1} together with \eqref{A12} implies that  $I_{\omega}(v_{n})\rightarrow 0$ as $n$  goes to $+\infty$. Next, 
let us set  $f_{n}(x)=\rho_{n}v_{n}(x)$ with
\begin{equation*}
\rho_{n}=\exp\left(\frac{I_{\omega}(v_{n})}{2\|v_{n}\|^{2}_{L^{2}}}\right),
\end{equation*}
where $\exp(x)$ represent the exponential function. We know  that $\rho_{n}\rightarrow 1$ as $n$  goes to $+\infty$, and $I_{\omega}(f_{n})=0$ for any $n\in\mathbb{N}$.  Since  $\left\{v_{n}\right\}$  is bounded in $\Sigma(\mathbb{R}^{N})$, it follows immediately that $\|v_{n}-f_{n}\|_{\Sigma}\rightarrow 0$. By virtue of  \eqref{A12}, we see that $\left\{f_{n}\right\}$ is a minimizing sequence for $d(\omega)$. Thanks to Theorem \ref{ESSW} we know that, up to a subsequence, there exists $\theta_{0}\in \mathbb{R}$  such that
\begin{equation}\label{UEP1}
\| f_{n}- e^{i\theta_{0}}\phi_{\omega}\|_{\Sigma}\rightarrow 0, \quad \text{as $n\rightarrow +\infty$}.
\end{equation}
Thus,  in view of the triangular inequality, \eqref{UEP1} and remembering that $v_{n}= u_{n}(t_{n})$, one can easily proves that
\begin{equation*}
\left\|u_{n}(t_{n})-e^{i\theta_{0}}\phi_{\omega}\right\|_{\Sigma}\rightarrow 0\quad \text{as $n\rightarrow +\infty$},
\end{equation*}
which is a contradiction with \eqref{3C2}. This completes the proof of the orbital stability of the ground states of \eqref{0NL}. 
\end{proof}

\bigskip


\bigskip
\medskip

\begin{thebibliography}{10}


\bibitem{AHAA}
J.~Angulo and A.~H. Ardila.
\newblock Stability of standing waves for logarithmic {S}chrödinger equation
  with attractive delta potential.
\newblock {\em  Indiana Univ. Math. J.}  67(2),  471--494, 2018.

\bibitem{AHA1}
A.H. Ardila.
\newblock Orbital stability of gausson solutions to logarithmic {S}chrödinger
  equations.
\newblock {\em Electron. J. Diff. Eqns.}, 2016(335):1--9, 2016.

\bibitem{ArdSqu}
A.H. Ardila and M. Squassina.
\newblock
Gausson dynamics for logarithmic Schr\"odinger equations.
\newblock
{\em Asymptotic Anal.}, 107(3-4),  203--226, 2018.

\bibitem{CAS}
I.~Bialynicki-Birula and J.~Mycielski.
\newblock Nonlinear wave mechanics.
\newblock {\em Ann. Phys.}, 100:62--93, 1976.

\bibitem{PHBJ}
P.~H. Blanchard, J.~Stubbe, and L.~Vázquez.
\newblock On the stability of solitary waves for classical scalar fields.
\newblock {\em Ann. Inst. Henri-Poncaré, Phys. Théor.}, 47:309--336, 1987.

\bibitem{PHST}
Ph. Blanchard and J.~Stubbe.
\newblock Stability of ground states for nonlinear classical field theories.
\newblock volume 347 of {\em Lecture Notes in Physics}, pages 19--35. Springer
  Heidelberg, 1989.

\bibitem{BBA}
B.~Bouharia.
\newblock Stability of logarithmic {B}ose-{E}instein condensate in harmonic
  trap.
\newblock {\em Modern Physics Letters B}, 29(01):1450260, 2015.

\bibitem{LBL}
H.~Brézis and E.~Lieb.
\newblock A relation between pointwise convergence of functions and convergence
  of functionals.
\newblock {\em Proc. Amer. Math. Soc.}, 88(3):486--490, 1983.

\bibitem{Carles2002}
R.~Carles.
\newblock Remarks on the nonlinear {S}chrödinger equation with harmonic
  potential.
\newblock {\em Ann. Henri Poincaré}, 3:757--772, 2002.

\bibitem{CL}
T.~Cazenave.
\newblock Stable solutions of the logarithmic {\Large s}chrödinger equation.
\newblock {\em Nonlinear. Anal., T.M.A.}, 7:1127--1140, 1983.

\bibitem{CB}
T.~Cazenave.
\newblock {\em Semilinear {S}chrödinger Equations}.
\newblock Courant Lecture Notes in Mathematics,10. American Mathematical
  Society, Courant Institute of Mathematical Sciences, 2003.

\bibitem{TA}
T.~Cazenave and A.~Haraux.
\newblock Equations d'évolution avec non-linéarité logarithmique.
\newblock {\em Ann. Fac. Sci. Toulouse Math.}, 2(1):21--51, 1980.

\bibitem{CALO}
T.~Cazenave and P.L. Lions.
\newblock Orbital stability of standing waves for some nonlinear {S}chrödinger
  equations.
\newblock {\em Comm. Math. Phys.}, 85(4):549--561, 1982.

\bibitem{Fukui2000}
R.~Fukuizumi.
\newblock Stability and instability of standing waves for the schrödinger
  equation with harmonic potential.
\newblock {\em Discrete Contin. Dynam. Systems}, 7:525--544, 2000.

\bibitem{FOG}
R.~Fukuizumi.
\newblock Stability of standing waves for nonlinear {\Large s}chrödinger
  equations with critical power nonlinearity and potentials.
\newblock {\em Advances in Differential Equations}, 10(2):259--276, 2005.

\bibitem{FUKUIOHA2003}
R.~Fukuizumi and M.~Ohta.
\newblock Stability of standing waves for nonlinear {S}chrödinger equations
  with potentials.
\newblock {\em Differential and Integral Equations}, 16(1):111--128, 2003.

\bibitem{AH}
A.~Haraux.
\newblock {\em Nonlinear Evolution Equations: Global Behavior of Solutions},
  volume 841 of {\em Lecture Notes in Math.}
\newblock Springer-Verlag, Heidelberg, 1981.

\bibitem{HE}
E.F. Hefter.
\newblock Application of the nonlinear schrödinger equation with a logarithmic
  inhomogeneous term to nuclear physics.
\newblock {\em Phys. Rev}, 32(A):1201--1204, 1985.

\bibitem{CZ}
C.~Ji and A.~Szulkin.
\newblock A logarithmic {S}chr{\"o}dinger equation with asymptotic conditions
  on the potential.
\newblock {\em J. Math. Anal. Appl.}, 437(1):241--254, 2016.

\bibitem{ELL}
E.~Lieb and M.~Loss.
\newblock {\em Analysis}, volume~14 of {\em Graduate Studies in Mathematics}.
\newblock American Mathematical Society, Providence, RI, 2 edition, 2001.

\bibitem{MFG}
S.~De Martino, M.~Falanga, C.~Godano, and G.~Lauro.
\newblock Logarithmic {S}chrödinger-like equation as a model for magma
  transport.
\newblock {\em Europhys}, (Lett.63):472--475, 2003.


\bibitem{JZ2000}
J.~Zhang.
\newblock Stability of standing waves for nonlinear {S}chrödinger equations
  with unbounded potentials.
\newblock {\em Z. Angew. Math. Phys.}, 51:498--503, 2000.

\bibitem{JZ2SS005}
J.~Zhang.
\newblock Sharp threshold for global existence and blowup in nonlinear
  {S}chrödinger equation with harmonic potential.
\newblock {\em Commun. Partial Differ. Equ.}, 30:1429--1443, 2005.

\bibitem{APLES}
K.G. Zloshchastiev.
\newblock Logarithmic nonlinearity in theories of quantum gravity: {O}rigin of
  time and observational consequences.
\newblock {\em Grav. Cosmol.}, 16(4):288--297, 2010.

\bibitem{ADL}
K.G. Zloshchastiev.
\newblock Spontaneous symmetry breaking and mass generation as built-in
  phenomena in logarithmic nonlinear quantum theory.
\newblock {\em Acta Phys. Polon. B}, 42(261), 2011.


\end{thebibliography}
\end{document}